\documentclass[11pt]{amsart}

\usepackage[usenames,dvipsnames,svgnames,table]{xcolor}
\usepackage[hyphens]{url}
\usepackage[pagebackref,linktocpage=true,colorlinks=true,linkcolor=Blue,citecolor=BrickRed,urlcolor=RoyalBlue]{hyperref}
\usepackage[msc-links,abbrev]{amsrefs}
\usepackage{amsmath,amsthm,amssymb}

\newcommand{\R}{\mathbf{R}}
\renewcommand{\S}{\mathbf{S}}

\renewcommand{\epsilon}{\varepsilon}
\newcommand{\cl}{\operatorname{cl}}
\newcommand{\ol}{\overline}
\newcommand{\C}{\mathcal{C}}
\newcommand{\inte}{\operatorname{int}}
\newcommand{\vspan}{\operatorname{span}}
\newcommand{\Hess}{\operatorname{Hess}}
\newcommand{\dist}{\operatorname{dist}}

\theoremstyle{plain}
\newtheorem{theorem}{Theorem}[section]
\newtheorem{lemma}[theorem]{Lemma}
\newtheorem{proposition}[theorem]{Proposition}
\newtheorem{corollary}[theorem]{Corollary}

\theoremstyle{definition}

\begin{document}

\title[Mean curvatures  and symmetry]{Mean curvatures and symmetry\\of convex hypersurfaces}

\author{Mohammad Ghomi}
\address{School of Mathematics, Georgia Institute of Technology,
Atlanta, GA 30332}
\email{ghomi@math.gatech.edu}
\urladdr{ghomi.math.gatech.edu}

\begin{abstract}
Let $M^n$ be a $\C^2$ closed convex hypersurface in Euclidean space, and 
$\sigma_m$ be its $m$th mean curvature. We show that $M$ is 
symmetric with respect to a hyperplane orthogonal to a given direction 
$e$, if $\sigma_m(p)\leq\sigma_m(q)$ whenever $p-q$ is parallel to $e$.
For convex hypersurfaces, this settles a conjecture of Li, extends the mean-curvature theorem
of Li--Yan--Yao to all $\sigma_m$, and strengthens some earlier results of
Li--Nirenberg by removing  nondegeneracy assumptions. The proof is based on the theory of mixed volumes, specifically the rigidity of  quermassintegrals under Steiner symmetrization.

\end{abstract}

\date{\today\,(Last Typeset)}

\subjclass[2020]{Primary 52A20, 53C45; Secondary 52A39, 52A40.}
\renewcommand{\subjclassname}{\textup{2020} Mathematics Subject Classification}

\thanks{The author was supported by NSF grant DMS-2202337.}

\keywords{Generalized mean curvatures, Gauss--Kronecker curvature,
Quermassintegral, Continuous Steiner symmetrization, Mixed volume, Shadow system, Alexandrov reflection method.}

\maketitle

\section{Introduction}
Alexandrov \cite{Ale62} showed that a closed embedded 
hypersurface $M\subset\R^{n+1}$ is a sphere if its mean curvature $\sigma_1$ is constant, and Ros \cite{ros1987,montiel-ros1991} extended that result to higher mean curvatures $\sigma_m$, $1\leq m\leq n$.
Following earlier work of Li \cite{li1997},
Li--Nirenberg \cites{li-nirenberg2006a,li-nirenberg2006b} 
studied a refinement of Alexandrov's theorem: must $M$ be symmetric with respect to a hyperplane orthogonal to a direction $e$ if $\sigma_1$ is suitably ordered along $e$? 
Li--Yan--Yao \cite{li-yan-yao2021} answered this question affirmatively  when $M$ lies on one side of its tangent hyperplanes parallel to $e$, or under a substantially weaker condition. Li  then conjectured a similar phenomenon for $\sigma_m$ \cite[Conj.~2]{li2023}, and noted that it is open even for convex $M$. 
 Here we resolve that case:

\begin{theorem}\label{thm:main}
Let $M\subset\R^{n+1}$ be a $\C^2$ hypersurface bounding a convex body $K$, and $e\in\S^n$ be a direction. Suppose that, for some $1\leq m\leq n$,
\begin{equation}\label{eq:sigma}
\sigma_m(p)\leq\sigma_m(q)
\end{equation}
whenever $p-q=\lambda e$, for some $\lambda\geq 0$, and the open
segment joining $p$ and $q$ lies in the interior of $K$. Then $M$ is symmetric with respect to a hyperplane orthogonal to $e$.
\end{theorem}

For $\sigma_1$, the above theorem was first established by Li--Nirenberg \cite[Cor.~1~\&~2]{li-nirenberg2006b}, assuming analyticity or a nondegeneracy condition, and was also proved for more general curvature functions including Gauss--Kronecker curvature $\sigma_n>0$  \cite[Thm. 2]{li-nirenberg2006b}, under more regularity. Theorem \ref{thm:main} allows $\sigma_n$ to vanish.  Our proof is variational, in contrast with the Alexandrov reflection method  and the Hopf lemma employed by Li--Nirenberg \cites{li-nirenberg2006a,li-nirenberg2006b}. As suggested by Li \cite{li2023}, and in line with the approach of Li--Yan--Yao \cite{li-yan-yao2021} for mean curvature $\sigma_1$, the proof of Theorem \ref{thm:main} exploits the variational properties of $\sigma_m$, which we access through the theory of mixed volumes. 

More precisely, we study the variation $K_t$ of convex bodies,  produced by continuous Steiner symmetrization of $K$. This process moves the centers of the chords of $K$ parallel to $e$ into the hyperplane $e^\perp$ with constant speed, while preserving their lengths, to obtain the symmetric body $K_0$, and continues these movements until we reach $K_{-1}$, the reflection of $K$ through $e^\perp$. The first variation formula for quermassintegrals
shows that equality holds in \eqref{eq:sigma}.
Consequently, the derivative of  the quermassintegral $W_m(K_t)$ vanishes at $t=1$, which is first proved in the strictly convex case, and then by approximation in the general case; see also Appendix \ref{appendix:formula} for a direct argument. Since $W_m(K_t)$ is convex and takes equal values at $t=\pm 1$, it must then be constant. In particular, $W_m(K_0)=W_m(K)$. Thus Theorem \ref{thm:main} is reduced to the following rigidity result  for quermassintegrals \cite[Zusatz IV, p.~261]{hadwiger1957}:

\begin{theorem}[Hadwiger \cite{hadwiger1957}]\label{thm:rigidity}
Let $K\subset\R^{n+1}$ be a convex body, and $K_0$ be its Steiner 
symmetrization with respect to the hyperplane $e^\perp$, for some direction
$e\in\S^n$. Then, for $1\leq m\leq n$,
\begin{equation}\label{eq:Wm}
W_m(K_0)\leq W_m(K),
\end{equation}
with equality only if $K$ is symmetric with respect to a hyperplane 
orthogonal to $e$.
\end{theorem}

Inequality \eqref{eq:Wm} generalizes the isoperimetric inequality $W_m(B)\leq W_m(K)$  for a ball $B$ of the same volume as $K$, which follows from the Alexandrov--Fenchel inequality for mixed volumes \cite{schneider2014}*{p.~401}; see  also \cite{milman-yehudayoff2023}*{Thm.~1.1} and the surrounding discussion. The main point of these inequalities here is the rigidity of the equality case. In this sense, the relation between Theorem \ref{thm:rigidity} and the inequality $W_m(B)\leq W_m(K)$ mirrors that between Theorem \ref{thm:main} and Alexandrov's sphere theorem.  See Milman--Yehudayoff \cite{milman-yehudayoff2023}*{Thm. 1.4} for the affine analogue of Theorem \ref{thm:rigidity}, and Chleb\'{\i}k--Cianchi--Fusco \cite{chlebik-cianchi-fuscof2005} for an extension to sets of finite perimeter when $m=1$.  The proof of Theorem \ref{thm:main} is presented in Section \ref{sec:symmetry}, after we develop the required background in Sections \ref{sec:preliminaries} and \ref{sec:steiner}. We also include a proof of Theorem \ref{thm:rigidity} in Appendix \ref{sec:appendix}.

\section{Preliminaries: Mixed Volumes}\label{sec:preliminaries}
We start by gathering the basic facts we need from the theory of mixed volumes, which may be found in the standard reference \cite{schneider2014}. A \emph{convex body} $K\subset\R^{n+1}$ is a compact convex set with nonempty interior. These bodies form a metric space with respect to Hausdorff distance.
For convex bodies $K,L$ and $\lambda\geq0$, let 
$$
K+L:=\big\{x+y\mid x\in K,\ y\in L\big\}, \qquad\text{and}\qquad \lambda K:=\big\{\lambda x\mid x\in K\big\}
$$ 
be the \emph{Minkowski addition} and \emph{scalar multiplication} respectively. Write $|\cdot|$ for volume, or Lebesgue measure, in the dimension of the given set, and  let $B:=B^{n+1}\subset\R^{n+1}$ be the unit ball. The \emph{quermassintegrals} $W_0(K),\dots,W_{n+1}(K)$ of a convex body $K$ are defined by Steiner's polynomial
\begin{equation}\label{eq:steiner}
|K+\lambda B|
=
\sum_{m=0}^{n+1}\binom{n+1}{m}W_m(K)\lambda^m.
\end{equation}
More generally, Minkowski's polynomial formula defines the \emph{mixed volumes} $V$ of a collection of convex bodies $K_1,\dots,K_k$ by
$$
\big|\lambda_1K_1+\cdots+\lambda_kK_k\big|
=
\sum_{i_1,\dots,i_{n+1}=1}^k
\lambda_{i_1}\cdots\lambda_{i_{n+1}}
V(K_{i_1},\dots,K_{i_{n+1}}).
$$
The mixed volume is symmetric, continuous under Hausdorff convergence, and linear in each argument with respect to Minkowski addition and scalar multiplication. Let $K[r]$ denote $r$ repeated arguments equal to $K$. Comparing the two polynomial formulas above gives, for $0\leq m\leq n+1$, 
\begin{equation}\label{eq:WmK}
W_m(K)=V\Big(K[\ol{m}],B[m]\Big),\qquad \ol{m}:=n+1-m.
\end{equation}

Let $h_K\colon\S^n\to\R$ be the \emph{support function} of $K$ given by $h_K(\nu):=\max_{p\in K}\langle p,\nu\rangle$. For convex bodies $K_1,\dots,K_n$, the \emph{mixed area measure} $S(K_1,\dots,K_n,\cdot)$ is the unique finite Borel measure on $\S^n$ characterized by
\begin{equation}\label{eq:mixed-area}
V(L,K_1,\dots,K_n)
=
\frac{1}{n+1}\int_{\S^n}h_L\,dS(K_1,\dots,K_n)
\end{equation}
for every convex body $L$. It is symmetric and weakly continuous in its arguments. For brevity, set 
$$
S_m(K,\cdot):=S\Big(K[n-m],\,B[m],\;\cdot\;\Big),
$$
for $0\leq m\leq n$. We need the following variation formula, which is a version of a formula of
Alexandrov \cite{schneider2014}*{Lem.~7.5.3}. 

\begin{lemma}\label{lem:variation}
Let $K_t\subset\R^{n+1}$, $0\leq t\leq 1$, be a one-parameter family of convex bodies with support functions $h_t$, and $K_1=K$. Suppose that the difference quotients $(h_t-h_1)/(t-1)$ converge uniformly on $\S^n$ as $t\to1^-$. Then, for $0\leq m\leq n$,
$$
\left.\frac{d}{dt}\right|_{t=1^-}W_m(K_t)
=
\frac{\ol{m}}{n+1}\int_{\S^n}\left.\frac{d}{dt}\right|_{t=1^-}\!h_t\;dS_m(K).
$$
\end{lemma}

\begin{proof}
Using \eqref{eq:WmK} and telescoping in the first $\ol{m}$ arguments,
$$
W_m(K_t)-W_m(K)
=
\sum_{i=0}^{\ol{m}-1}
\Bigl(V\Bigl(K_t[\ol{m}-i],K[i],B[m]\Bigr)-V\Bigl(K_t[\ol{m}-1-i],K[i+1],B[m]\Bigr)\Bigr).
$$
Applying \eqref{eq:mixed-area} to each mixed volume in the sum, the $i$th summand equals
$$
\frac{1}{n+1}
\int_{\S^n}(h_t-h_1)\,dS\Big(K_t[\ol{m}-1-i],K[i],B[m]\Big).
$$
Now divide by $t-1$ and let $t\to1^-$. Since $h_t\to h_1$ uniformly by hypothesis, $K_t\to K$ in the Hausdorff metric. So, by weak continuity, the mixed area measures above converge to $S(K[\ol{m}-1],B[m],\cdot)=S_m(K,\cdot)$; in particular their total masses are bounded. Furthermore, the difference quotients converge uniformly to the derivative $\frac{d}{dt}\big|_{t=1^-}h_t$, which is continuous, being a uniform limit of continuous functions. Consequently, each of the $\ol{m}$ summands, divided by $t-1$, converges to $\int_{\S^n}\frac{d}{dt}\big|_{t=1^-}h_t\,dS_m(K)/(n+1)$, which yields the desired formula.
\end{proof}

We say that a convex body $K$ is $\C^k$ if  its boundary $M$ is a $\C^k$ hypersurface.
Let $K$ be $\C^2$ with outward Gauss map $N$, and principal curvatures $\kappa_1,\dots,\kappa_n\geq0$. 
The $m$th \emph{mean curvature} of $M$, for  $1\leq m\leq n$, is given by the elementary symmetric function
$$
\sigma_m:=\sum_{1\leq i_1<\cdots<i_m\leq n}
\kappa_{i_1}\cdots\kappa_{i_m}.
$$
The curvature $\sigma_1$, or its normalization $\sigma_1/n$, is  known as the \emph{mean curvature}, and $\sigma_n$ is called the \emph{Gauss--Kronecker curvature}.
Let $dA$ denote the area element of $M$. Then
\begin{equation}\label{eq:curvature-representation}
\int_M (f\circ N)\,\sigma_m\,dA
=
\binom{n}{m}\int_{\S^n} f\,dS_m(K,\cdot),
\end{equation}
for all continuous functions $f$ on $\S^n$. In other words, $\sigma_m\,dA$ is the pullback of $\binom{n}{m}S_m(K,\cdot)$ under  $N$. Setting $f\equiv1$ and using  \eqref{eq:mixed-area} with $L=B$, we obtain
$$
\int_M\sigma_m\,dA
=
(n+1)\binom{n}{m}W_{m+1}(K).
$$
Thus, when $K$ is $\C^2$, the quermassintegral $W_{m+1}(K)$ is, up to a dimensional factor, the total $\sigma_m$-curvature of $M$.

We close with another interpretation for quermassintegrals. For a linear subspace $V\subset\R^{n+1}$, let $P_V\colon\R^{n+1}\to V$ denote the orthogonal projection onto $V$. For a nonzero vector $v\in\R^{n+1}$, we write $P_v:=P_{\vspan(v)}$. 
Let $Gr(n+1,\ol{m})$ be the Grassmannian of $\ol{m}$-dimensional linear subspaces $E\subset\R^{n+1}$, equipped with its rotation-invariant probability measure $dE$. The \emph{Cauchy--Kubota projection formula} states that
\begin{equation}\label{eq:CK}
W_m(K)
=
\frac{|B^{n+1}|}{|B^{\ol{m}}|}
\int_{Gr(n+1,\ol{m})}|P_EK|\,dE,
\end{equation}
for $0\leq m\leq n+1$. In particular,  $W_0=|K|$, and $W_{n+1}(K)=|B^{n+1}|$, which carry no variational information. For the rest of this 
work we assume that $1\leq m\leq n$.

\section{The Symmetry Variation}\label{sec:steiner}
Here we develop a variation of convex bodies via Steiner symmetrization. This device goes back
to P\'olya--Szeg\H{o} \cite[Note B]{polya-szego1951}, with 
subsequent variants \cite{rogers-shephard1958,shephard1964,brock1995},  resulting in numerous applications, see \cite[Sec.~1.1]{milman-shabelman-yehudayoff2025}, \cite[Sec. 10.4]{schneider2014},  \cite{gpsy2021,hwxy2026}. 
Let $K\subset\R^{n+1}$ be a convex body, $e\in\S^n$, and $e^\perp$ be the subspace of $\R^{n+1}$ orthogonal to $e$.  After a rotation we may assume that $e$ is the last coordinate direction,  so $e^\perp=\R^n\times\{0\}$. Let 
$$
\pi:=P_{e^\perp}\colon\R^{n+1}\to \R^n\times\{0\}
$$ 
be projection into the first $n$ coordinates. Set $\Omega:=\inte(\pi K)$, so that $\cl(\Omega)=\pi K$, where \emph{int} and \emph{cl} stand for interior and closure respectively. We denote points of $\R^{n+1}$ by $p=(x,z)\in \R^n\times \R$. Now write
$$
K=\left\{(x,z)\mid x\in\cl(\Omega),\ u^-(x)\leq z\leq u^+(x)\right\},
$$
where $u^+$ is a concave and $u^-$ is a convex function on $\cl(\Omega)$. By a \emph{chord} of $K$ (parallel to $e$) over a point $x\in\cl(\Omega)$ we mean the vertical fiber $\pi^{-1}(x)\cap K$.
The radius (or half the length) and center of these chords are given by the functions
$$
r:=(u^+-u^-)/2,\qquad\text{and}\qquad c:=(u^++u^-)/2.
$$
For $-1\leq t\leq 1$, we define the \emph{symmetry variation} of $K$ with respect to $e$ as
\begin{equation}\label{eq:Kt}
K_t:=\bigl\{p+(t-1)c(\pi p)e\mid p\in K\bigr\}.
\end{equation}
This map translates the chord of $K$ over each $x\in\cl(\Omega)$ by $(t-1)c(x)e$. Hence the chord of $K_t$ over $x$ is given by
\begin{equation}\label{eq:Kt-chords}
\pi^{-1}(x)\cap K_t=\bigl\{(x,z)\mid -r(x)+tc(x)\leq z\leq r(x)+tc(x)\bigr\},
\end{equation}
which has length $2r(x)$ and center $tc(x)$. In particular $K_1=K$, $K_{-1}$ is the reflection of $K$ in $e^\perp$, and $K_0$ is the Steiner symmetrization of $K$ with respect to $e^\perp$. Each $K_t$ is convex, since the upper endpoint of its chords, 
$r+tc=(1+t)u^+/2-(1-t)u^-/2$, is concave, while the lower 
endpoint, $-r+tc=(1+t)u^-/2-(1-t)u^+/2$, is convex.  Set
$$
W_m(t):=W_m(K_t),\qquad\text{and}\qquad A_E(t):=|P_EK_t|,
$$
where recall that $\ol{m}:=n+1-m$, and  $E\in Gr(n+1,\ol{m})$ as in the last section.
Then $W_m(1)=W_m(-1)$. The Cauchy--Kubota formula \eqref{eq:CK} may now be written as
\begin{equation}\label{eq:CK2}
W_m(t)
=
\frac{|B^{n+1}|}{|B^{\ol{m}}|}
\int_{Gr(n+1,\ol{m})}A_E(t)\,dE.
\end{equation}

To exploit this formula, we deploy the following fibration mechanism.
Let $E\in Gr(n+1,\ol{m})$ with $\ol e:=P_Ee\ne0$, and set $\ol E:=E\cap e^\perp$.
Then we have the orthogonal decomposition
$$
E=\ol E\oplus\vspan\{\ol e\}.
$$ 
Thus $P_EK_t$ may be fibered by line segments $F_\xi$ parallel to $\ol e$, over points $\xi\in P_{\ol E}\cl(\Omega)$. Let $\ell_\xi(t)$ denote the length of $F_\xi$. Then, we have 
\begin{equation}\label{eq:fubini}
A_E(t)
=
\int_{P_{\ol E}\Omega}\ell_\xi(t)\,d\xi,
\qquad\quad
\ell_\xi(t)=\Big(L^+_\xi(t)-L^-_\xi(t)\Big)|\ol e|,
\end{equation}
where the first equality is by Fubini's theorem, and
$L^{\pm}_\xi(t)$ are the upper and lower $\ol e$-coordinates of $F_\xi$ in $E$. 

\begin{lemma}\label{lem:fibration}
Let $E\in Gr(n+1,\ol{m})$ with $\ol e:=P_Ee\ne0$. Then
$$
L^+_\xi(t)
=
\sup_{\substack{x\in\cl(\Omega)\\P_{\ol E}x=\xi}}
\left\{
\frac{\langle x,\ol e\rangle}{|\ol e|^2}+r(x)+tc(x)
\right\},
\quad
L^-_\xi(t)
=
\inf_{\substack{x\in\cl(\Omega)\\P_{\ol E}x=\xi}}
\left\{
\frac{\langle x,\ol e\rangle}{|\ol e|^2}-r(x)+tc(x)
\right\}.
$$
In particular, $L^+_\xi$ is convex and $L^-_\xi$ is concave in $t$, and, for each fixed $t$, the function $\xi\mapsto\ell_\xi(t)$ is continuous on $\inte(P_{\ol E}\Omega)$. 
\end{lemma}

\begin{proof}
For $x\in\cl(\Omega)$, set $\xi_x:=P_{\ol E}x$. Since $\langle e,\ol e\rangle=\langle P_Ee,P_Ee\rangle=|\ol e|^2$, we have $P_{\ol e}e=\ol e$. Then
$$
P_E(x+ze)
=
P_{\ol E}(x+ze)+P_{\ol e}(x+ze)
=
P_{\ol E}x+P_{\ol e}x+z\ol e
=
\xi_x+
\left(
\frac{\langle x,\ol e\rangle}{|\ol e|^2}+z
\right)\ol e.
$$
Thus $P_E$ maps the chord of $K_t$ over each $x$ into the fiber of $P_EK_t$ over $\xi_x$. Hence \eqref{eq:Kt-chords} yields the expressions for $L^\pm_\xi(t)$. In particular, $L^+_\xi$ is convex in $t$, as a supremum of affine functions, while $L^-_\xi$ is concave, as an infimum of affine functions. Furthermore, for each fixed $t$, $L^+_\xi(t)$ and $L^-_\xi(t)$ are finite concave and convex functions of $\xi$, respectively, on $P_{\ol E}\Omega$, and hence are continuous in its interior; so $\xi\mapsto\ell_\xi(t)$ is continuous there. 
\end{proof}

\begin{corollary}\label{cor:convexity}
For every $E\in Gr(n+1,\ol{m})$, the functions $A_E(t)$ and $W_m(t)$ are convex.
\end{corollary}

\begin{proof}
By \eqref{eq:CK2} it suffices to show that $A_E(t)$ is convex, since averages of convex functions are convex. If $P_Ee=0$, then $A_E(t)=A_E(0)$ for every $t$, and we are done. Otherwise, by Lemma~\ref{lem:fibration}, each $\ell_\xi$ is convex in $t$, being the difference of a convex and a concave function, up to a positive factor; hence so is $A_E(t)$ by \eqref{eq:fubini}.
\end{proof}

The family $K_t$ may also be viewed as a \emph{shadow system} in the sense of Shephard \cite{shephard1964}, since it arises from the projections of a convex body in $\R^{n+2}$, as may be seen from the identity $K_t=T_t(\Delta(K))$ in the proof of Lemma~\ref{lem:deformation-continuity} below. So the convexity of $W_m(t)$ also follows from the convexity of mixed volumes of shadow systems established in \cite{shephard1964}.

\section{Proof of Theorem \ref{thm:main}}\label{sec:symmetry}
Here we complete the proof of Theorem~\ref{thm:main} in three stages as follows. Let $K$ be the convex body bounded by $M$, let $e$ be the last coordinate direction, let $\Omega$, $u^{\pm}$, $r$, $c$, and $K_t$ be as in Section~\ref{sec:steiner},  let $\sigma_m^{\pm}$ denote the $\sigma_m$-curvatures of the graphs of $u^{\pm}$ over $\Omega$, and let $N$ be the outward Gauss map of $M$. 

\subsection{Reduction to equality}
The first variation formula established for quermassintegrals in Lemma \ref{lem:variation}  quickly yields that equality holds in \eqref{eq:sigma}:

\begin{lemma}\label{lem:curvature-equality}
$\sigma_m^+=\sigma_m^-.$
\end{lemma}

\begin{proof}
Set
$
L_t:=K+(t-1)e
$
for $0\leq t\leq1$. If $h_t$ is the support function of $L_t$, then
$
h_t(\nu)=h_K(\nu)+(t-1)\langle e,\nu\rangle.
$
Since $L_t$ is a translate of $K$, $W_m(L_t)$ is constant. Thus
Lemma~\ref{lem:variation} and \eqref{eq:curvature-representation} give
$$
0
=
\frac{d}{dt}W_m(L_t)
=
\frac{\ol{m}}{(n+1)\binom{n}{m}}
\int_M\langle e,N\rangle\,\sigma_m\,dA.
$$
Over the graphs of $u^\pm$ we have
$dx=\pm\langle e,N\rangle\,dA$, while $\langle e,N\rangle=0$ on the
remainder of $M$. Consequently,
$
0
=
\int_\Omega\bigl(\sigma_m^+-\sigma_m^-\bigr)\,dx.
$
For each $x\in\Omega$, the points $p:=(x,u^+(x))$ and $q:=(x,u^-(x))$
satisfy the hypotheses of \eqref{eq:sigma}. Hence
$\sigma_m^+-\sigma_m^-\leq 0$ on $\Omega$. Since
$\sigma_m^+-\sigma_m^-$ is continuous, it
vanishes identically. 
\end{proof}

The identity $\int_M\langle e,N\rangle\,\sigma_m\,dA=0$  above
also follows from the fact that mixed area measures have centroid at the origin, cf. \cite{schneider2014}*{(5.30)}. Moreover, since convex hypersurfaces satisfy Condition~S' of \cite{li-yan-yao2021}, 
the above lemma  follows from
\cite[Prop. 3]{li-yan-yao2021} when $m=1$, and \cite[Lem. 1]{li2023} when $2\leq m\leq n$.

\subsection{The strictly convex case}

The next lemma computes the derivative of the support functions of the symmetry variation $K_t$ at $t=1$ when the original body $K=K_1$ is
\emph{strictly convex}, meaning that its boundary $M$ contains no nontrivial line segment.

\begin{lemma}\label{lem:supportvariation}
Suppose that $K$ is strictly convex. Let $h_t$ be the support function of $K_t$. For each $\nu\in\S^n$, let $p_\nu$ be the unique point of $K$ with outward support normal $\nu$. Then
$$
\left.\frac{d}{dt}\right|_{t=1^-}h_t(\nu)
=
c(\pi p_\nu)\langle e,\nu\rangle.
$$
Furthermore, the corresponding difference quotients $(h_t(\nu)-h_1(\nu))/(t-1)$ converge uniformly in $\nu$ as $t\to1^-$.
\end{lemma}
\begin{proof}
Strict convexity implies that the chord of $K$ over every point of $\partial\Omega$ is a single point, since it lies in a vertical supporting hyperplane of $K$. Let $x_i\in\Omega$ tend to $x\in\partial\Omega$. By compactness, every limit point of $(x_i,u^+(x_i))$ lies in the chord of $K$ over $x$, and hence equals its unique point. Thus $u^+(x_i)$ converges to the height of that point. The same argument applies to $u^-$, so $u^{\pm}$, and hence $c$, are continuous on $\cl(\Omega)$. By \eqref{eq:Kt},
$$
h_t(\nu)
=
\max_{p\in K}
\bigl\{\langle p,\nu\rangle+(t-1)c(\pi p)\langle e,\nu\rangle\bigr\}.
$$
Let $p_{t,\nu}$ be a maximizing point, and note 
that $h_1(\nu)=\langle p_\nu,\nu\rangle$. Taking $p=p_\nu$ in the maximum, and using the fact that  $\langle p_{t,\nu},\nu\rangle\leq h_1(\nu)$, since $p_{t,\nu}\in K$, yields
$$
h_t(\nu)\geq h_1(\nu)+(t-1)\,c(\pi p_\nu)\langle e,\nu\rangle,
\qquad
h_t(\nu)
\leq 
h_1(\nu)+(t-1)\,c(\pi p_{t,\nu})\langle e,\nu\rangle.
$$
Dividing these inequalities by $t-1<0$, we obtain
$$
c(\pi p_{t,\nu})\langle e,\nu\rangle
\leq
\frac{h_t(\nu)-h_1(\nu)}{t-1}
\leq
c(\pi p_\nu)\langle e,\nu\rangle.
$$
As $t\to1^-$, every limit point of $p_{t,\nu}$ is a support point of $K$ with normal $\nu$, and hence equals $p_\nu$. This convergence is uniform in $\nu$: otherwise a convergent sequence of normals and compactness of $K$ would produce two distinct support points for the limiting normal. Continuity of $c$ now gives the result.
\end{proof}

The next result covers in particular the case $\sigma_n>0$, established earlier by Li--Nirenberg \cite[Thm. 2]{li-nirenberg2006b} under stronger regularity assumptions. The following formula is an analogue, for the symmetry variation $K_t$, of the 
classical first variation formulas for curvature integrals of smooth 
families of hypersurfaces, e.g., see Reilly \cite{reilly1973}; however, $K_t$ is not a smooth variation of $K$, since $c$ 
is in general only continuous. Let $W'_m(1)$ denote the left derivative of $W_m$ at $t=1$. 

\begin{proposition}\label{prop:strict}
If $M$ is strictly convex, then
\begin{equation}\label{eq:key}
W_m'(1)
=
\frac{\ol{m}}{(n+1)\binom{n}{m}}\int_\Omega c(x)
\bigl(\sigma_m^+(x)-\sigma_m^-(x)\bigr)\,dx.
\end{equation}
Consequently, Theorem~\ref{thm:main} holds in the strictly convex case.
\end{proposition}

\begin{proof}
 By Lemmas~\ref{lem:variation} and \ref{lem:supportvariation},
$$
W_m'(1)
=
\frac{\ol{m}}{n+1}\int_{\S^n}c(\pi p_\nu)\langle e,\nu\rangle\,dS_m(K,\nu).
$$
We change variables twice. First, applying \eqref{eq:curvature-representation} with $f(\nu):=c(\pi p_\nu)\langle e,\nu\rangle$ yields
$$
\int_{\S^n}c(\pi p_\nu)\langle e,\nu\rangle\,dS_m(K,\nu)
=\frac{1}{\binom{n}{m}}
\int_M c(\pi p)\langle e,N(p)\rangle\,\sigma_m(p)\,dA(p),
$$
since $p_{N(p)}=p$. Next, over the graphs of $u^{\pm}$ we have $dx=\pm\langle e,N\rangle\,dA$, while $\langle e,N\rangle=0$ on the rest of $M$; hence, 
$$
\int_M c(\pi p)\langle e,N(p)\rangle\,\sigma_m(p)\,dA(p)
=
\int_\Omega c(x)
\bigl(\sigma_m^+(x)-\sigma_m^-(x)\bigr)\,dx,
$$
which yields the stated formula.
By
Lemma \ref{lem:curvature-equality},  $\sigma_m^+=\sigma_m^-$. So $W_m'(1)=0$.
By Corollary~\ref{cor:convexity}, $W_m(t)$ is convex, and we also have $W_m(-1)=W_m(1)$. Hence $W_m(t)$ is constant, and Theorem~\ref{thm:rigidity} yields the desired symmetry.
\end{proof}

Equation \eqref{eq:key} holds without the strict convexity assumption (see Appendix \ref{appendix:formula}), which completes the proof of Theorem \ref{thm:main} as discussed above. Here we give an alternative argument by approximation, which is conceptually simpler.

\subsection{The general case}
We need the following three lemmas. Given a convex body $K\subset\R^{n+1}$ with $0\in\inte K$, its boundary $M$ is the 
graph of the \emph{radial function} $\rho\colon\S^n\to(0,\infty)$ given 
by $\rho(u):=\max\{r>0\mid ru\in K\}$. For convex bodies $K_i$, $K$ containing $0$ in their 
interiors, with radial functions $\rho_i$, $\rho$, we say that $M_i\to M$ 
in the \emph{$\C^k$ topology} provided that $\rho_i\to\rho$ in 
$\C^k(\S^n)$.

\begin{lemma}\label{lem:approximation}
Let $K$ be a $\C^2$ convex body with  boundary $M$. Then there exist $\C^\infty$ convex bodies $K_i$ with  boundaries $M_i$ of positive principal curvatures, such that $M_i\to M$ in the $\C^2$ topology.
\end{lemma}

\begin{proof}
Translate $K$ so that $0\in\inte K$. Then the radial function $\rho$ of $K$ is $\C^2$.  Let $g\colon\R^{n+1}\to\R$ be 
the gauge function of $K$, given by $g(ru):=r/\rho(u)$ 
for $u\in\S^n$ and $r\geq0$. Then $g$ is convex, $K=\{g\leq1\}$, and $g$ 
is $\C^2$ on $\R^{n+1}\setminus\{0\}$. Let $\eta_i$ be smooth 
mollifiers with supports shrinking to the origin, let $\delta_i\to 0^+$, 
and set $g_i:=g*\eta_i+\delta_i|x|^2$. Then $g_i\to g$ in $\C^2$ on a 
fixed annular neighborhood of $M$, while $\Hess g_i\geq2\delta_iI$. 
For large $i$, $K_i:=\{g_i\leq1\}$ is a convex body 
with smooth boundary $M_i$. The second fundamental form of $M_i$ is given by
$\operatorname{II}_i(v,v)=\Hess g_i(v,v)/|\nabla g_i|>0$, 
so the principal curvatures of $M_i$ are positive. Since $g_i\to g$ in $\C^2$ near $M$, while 
$\partial_rg(\rho(u)u)=1/\rho(u)>0$, the implicit function theorem 
yields that $M_i\to M$ in the $\C^2$ topology.
\end{proof}

If $K$ is $\C^1$, let $N_K$ be its outward Gauss map, $c_K\colon P_{e^\perp}K\to\R$ denote the center function of its chords parallel to $e$, and define $f_K\colon M\to \R$ by
$$
f_K(p):=c_K(\pi p)\big\langle e,N_K(p)\big\rangle.
$$
This is the normal speed of $M$ under the symmetry 
variation \eqref{eq:Kt} at $t=1$.

\begin{lemma}\label{lem:normal-continuity}
Let $K_i$ and $K$ be $\C^1$ convex bodies with $K_i\to K$ in the Hausdorff metric. Suppose that $p_i\in\partial K_i$ converge to $p\in\partial K$. Then $f_{K_i}(p_i)\to f_K(p)$. In particular $f_K$ is continuous.
\end{lemma}

\begin{proof}
First we show that $N_{K_i}(p_i)\to N_K(p)$. Since these are unit vectors, it suffices to check that every convergent subsequence of $N_{K_i}(p_i)$ has limit $N_K(p)$. Suppose then that $N_{K_i}(p_i)\to\nu$. For each $x\in K$ there are $x_i\in K_i$ with $x_i\to x$, and $\langle x_i-p_i,N_{K_i}(p_i)\rangle\leq0$; letting $i\to\infty$ yields $\langle x-p,\nu\rangle\leq0$. So $\nu$ is the outward normal of a supporting hyperplane of $K$ at $p$. Since $\partial K$ is $\C^1$, its supporting hyperplane at $p$ is unique, and thus $\nu=N_K(p)$, as claimed.
Next note that, since all the bodies lie in a fixed ball,  $c_{K_i}$ are uniformly bounded. Hence the assertion is immediate if $\langle e,N_K(p)\rangle=0$. Suppose that $\langle e,N_K(p)\rangle\ne0$, and put $x_i=\pi p_i$, $x=\pi p$. Then $x\in\Omega=\inte(\pi K)$. Otherwise, a supporting hyperplane of $\cl(\Omega)$ at $x$ would lift to a vertical tangent hyperplane of $K$ at $p$,  contrary to $\langle e,N_K(p)\rangle\ne0$.
Write the upper and lower endpoints of the chords parallel to $e$ of $K_i$ and $K$ as $u^{\pm}_i$ and $u^{\pm}$, respectively. One quickly checks that  $u^\pm_i(x_i)\to u^\pm(x)$. Consequently $c_{K_i}(x_i)\to c_K(x)$. Hence $f_{K_i}(p_i)\to f_K(p)$.
Finally, taking $K_i=K$ shows that $f_K$ is continuous, and completes the proof.
\end{proof}

The last lemma is not trivial because $c_K$ may have discontinuities on $\partial\Omega$.

\begin{lemma}\label{lem:deformation-continuity}
If $K_i\to K$ in the Hausdorff metric, then $(K_i)_t\to K_t$ in the Hausdorff metric for every $-1\leq t\leq1$.
\end{lemma}

\begin{proof}
Let $H$ be the linear subspace of $(\R^{n+1})^2$ consisting of pairs with the same projection to $e^\perp$, and put $\Delta(K):=(K\times K)\cap H$. Since $H$ meets the interior of $K\times K$, Hausdorff convergence $K_i\times K_i\to K\times K$ implies $\Delta(K_i)\to\Delta(K)$.
Set $\alpha=(1+t)/2$ and $\beta=(1-t)/2$, and define on $H$ the linear map
$$
T_t\big((x,z_1),(x,z_2)\big)
:=
(x,\alpha z_1-\beta z_2).
$$
If the chord of $K$ over $x$ is $[u^-_K(x),u^+_K(x)]$, then the last coordinate of $T_t(\Delta(K))$ over $x$ ranges from $\alpha u^-_K(x)-\beta u^+_K(x)=-r_K(x)+tc_K(x)$ to $\alpha u^+_K(x)-\beta u^-_K(x)=r_K(x)+tc_K(x)$.
Thus $K_t=T_t(\Delta(K))$. Continuity of linear images in the Hausdorff metric now gives the assertion.
\end{proof}

Now we establish the main result of this work:

\begin{proof}[Proof of Theorem \ref{thm:main}]
Assuming $0\in \inte(K)$,
let $K_i$ and $M_i$ be as in Lemma~\ref{lem:approximation}, and $\rho_i$, $\rho$ denote the radial functions of $K_i$, $K$ respectively, so that $\rho_i\to\rho$ in $\C^2(\S^n)$; in particular, $K_i\to K$ in the Hausdorff metric.
Let $(K_i)_t$ be the symmetry variations of $K_i$, $W_{m,i}(t):=W_m((K_i)_t)$, and let $\sigma_{m,i}$ and $dA_i$ denote the $m$th mean curvature and area element of $M_i$. As in the proof of Proposition~\ref{prop:strict}, since $K_i$ is strictly convex, Lemmas~\ref{lem:variation} and \ref{lem:supportvariation}, together with \eqref{eq:curvature-representation}, yield 
$$
W_{m,i}'(1)
=
\frac{\ol{m}}{(n+1)\binom{n}{m}}
\int_{M_i}f_{K_i}\,\sigma_{m,i}\,dA_i.
$$

Parametrizing $M_i$ over $\S^n$ by $u\mapsto\rho_i(u)u$, the curvatures $\sigma_{m,i}$ and the area elements $dA_i$ depend continuously on $\rho_i$ up to second derivatives, and hence converge uniformly to the corresponding quantities for $M$. Furthermore, $f_{K_i}(\rho_i(u)u)\to f_K(\rho(u)u)$ uniformly, by Lemma~\ref{lem:normal-continuity} and compactness of $\S^n$. Consequently,
$$
\lim_{i\to\infty}W_{m,i}'(1)
=
\frac{\ol{m}}{(n+1)\binom{n}{m}}
\int_M f_K\,\sigma_m\,dA
=
\frac{\ol{m}}{(n+1)\binom{n}{m}}
\int_\Omega c_K
\bigl(\sigma_m^+-\sigma_m^-\bigr)\,dx
=0,
$$
where the second equality follows by change of variables as in the proof of Proposition~\ref{prop:strict}, and the last equality holds since $\sigma_m^+=\sigma_m^-$ by Lemma \ref{lem:curvature-equality}.
For $t<1$, convexity of $W_{m,i}$, which holds by Corollary~\ref{cor:convexity}, gives 
$$
W_{m,i}(t)\geq W_{m,i}(1)+W_{m,i}'(1)(t-1).
$$
 By Lemma~\ref{lem:deformation-continuity} and continuity of mixed volumes, $W_{m,i}(t)\to W_m(t)$ for every $t$. So letting $i\to\infty$ in the last  inequality gives $W_m(t)\geq W_m(1)$. On the other hand, $W_m(t)$ is convex, and we have $W_m(-1)=W_m(1)$. Hence $W_m(t)\leq W_m(1)$. Thus $W_m(t)$ is constant, and Theorem~\ref{thm:rigidity} completes the proof. 
\end{proof}

\appendix

\section{Rigidity of the Quermassintegrals}\label{sec:appendix}
Here we supply a proof of Theorem \ref{thm:rigidity}. Hadwiger's original 
argument \cite[pp.~256--261]{hadwiger1957} proceeds by induction on 
ambient dimension whereas ours is more direct and analytic. First we record the 
following observation, using the same notation as Section~\ref{sec:steiner}. 

\begin{lemma}\label{lem:affinity}
Suppose that $W_m(t)$ is constant. Then $A_E(t)$ is affine for every $E\in Gr(n+1,\ol{m})$. Furthermore, if $P_Ee\ne0$, then $\ell_\xi(t)$ is affine  for every $\xi\in\inte(P_{\ol E}\Omega)$.
\end{lemma}

\begin{proof}
Fix $E\in Gr(n+1,\ol{m})$, $0<\theta<1$, and $t_0,t_1\in[-1,1]$. Since $A_E(t)$ is convex by Corollary~\ref{cor:convexity}, the convexity deficit
$$
\delta(E)
:=
(1-\theta)A_E(t_0)+\theta A_E(t_1)-A_E\bigl((1-\theta)t_0+\theta t_1\bigr)
$$
is nonnegative. By \eqref{eq:CK2}, the average of $\delta$ over $Gr(n+1,\ol{m})$ is proportional to the corresponding deficit of $W_m$, which vanishes, as $W_m(t)$ is constant by assumption. Furthermore, $\delta$ is continuous, since $P_EK_t$ depends continuously on $E$, with respect to the Hausdorff metric, and hence so does its volume $A_E(t)$. Hence $\delta$ vanishes identically. Since $t_0$, $t_1$, and $\theta$ were arbitrary, $A_E$ is affine as claimed.

Next suppose that $P_Ee\ne0$, and fix $\xi\in\inte(P_{\ol E}\Omega)$. By Lemma~\ref{lem:fibration}, $\ell_\xi$ is convex in $t$, and continuous in $\xi$ on $\inte(P_{\ol E}\Omega)$ for each fixed $t$. So, for fixed $t_0,t_1$, and $\theta$ as above, the convexity deficit of $\ell_\xi$ is nonnegative and continuous in $\xi$, while its integral vanishes, by \eqref{eq:fubini}, because $A_E$ is affine. Hence the deficit vanishes for every interior $\xi$. Since $t_0,t_1$, and $\theta$ are arbitrary, $\ell_\xi$ is affine, which completes the proof.
\end{proof}

\begin{proof}[Proof of Theorem~\ref{thm:rigidity}]
Recall that $W_m(t):=W_m(K_t)$ and $K_1:=K$.
By Corollary~\ref{cor:convexity}, $W_m(t)$ is convex. Since 
$W_m(-1)=W_m(1)$, it follows that $W_m(0)\leq W_m(1)$. 
Suppose next that $W_m(0)=W_m(1)$. Then convexity forces $W_m(t)$ to be 
constant. It remains to show that $K$ is symmetric with 
respect to a hyperplane orthogonal to $e$.

Since $u^+$ is concave and $u^-$ is convex, the functions $r$ and $c$ are locally Lipschitz on $\Omega$ \cite{schneider2014}*{Thm.~1.5.3}, and have one-sided directional derivatives
$$
r'(x;v):=\lim_{\lambda\to0^+}\frac{r(x+\lambda v)-r(x)}{\lambda},
$$
and similarly $c'(x;v)$, at every point $x\in\Omega$ in every direction $v\in e^\perp$ \cite{schneider2014}*{Thm.~1.5.4}. It suffices to show that $c'(x;v)=0$. Indeed, then $c$ is constant along every line segment in $\Omega$, since it is Lipschitz. Hence, since $\Omega$ is connected, $c$ is constant, say $c\equiv z_0$. So reflection in the hyperplane $\langle p,e\rangle=z_0$ preserves the chords of $K$ over $\Omega$. Since $\inte(K)$ projects into $\Omega$, and $K=\cl(\inte(K))$, $K$ is symmetric.

Fix $x_0\in\Omega$, a direction $v_0\in e^\perp$, and an $m$-dimensional subspace $V\subset e^\perp$ containing $v_0$. We will show that $c'(x_0;v_0)=0$, which will complete the proof. We may choose $a\in V$ such that
\begin{equation}\label{eq:av}
r(x_0+v)\leq r(x_0)+\langle a,v\rangle,
\qquad\text{and}\qquad
\langle a,v_0\rangle=r'(x_0;v_0),
\end{equation}
for all $v\in V$ with $x_0+v\in\cl(\Omega)$. Indeed, since $u^+$ is concave and $u^-$ is convex, $r$ is concave on $\cl(\Omega)$. Hence $v\mapsto r'(x_0;v)$ is a finite superlinear function on $V$, i.e., concave and positively homogeneous \cite{schneider2014}*{Lem.~1.5.6}. Let $a\in V$ be a supergradient of this function at $v_0$, which exists since the function is finite and concave. So 
$$
r'(x_0;v)\leq r'(x_0;v_0)+\langle a,v-v_0\rangle
$$ 
for all $v\in V$. Setting $v=\lambda v_0$ and using positive homogeneity yields 
$(\lambda-1)r'(x_0;v_0)\leq(\lambda-1)\langle a,v_0\rangle$; taking 
$\lambda>1$ and $\lambda<1$ then gives $\langle a,v_0\rangle=r'(x_0;v_0)$, 
and consequently $r'(x_0;v)\leq\langle a,v\rangle$. Since the difference quotients of the concave function $r$ are nonincreasing, $r(x_0+v)-r(x_0)\leq r'(x_0;v)$, which establishes \eqref{eq:av}.

Now put $\ol E:=V^\perp\cap e^\perp$, $\ol e:=(e-a)/|e-a|^2$, and
$
E:=\ol E\oplus\vspan\{\ol e\}.
$
Then $\dim E=\ol{m}$. Since $a\in V$ and $a\perp e$, we have $|e-a|^2=1+|a|^2$, so $\ol e$ is well-defined and nonzero. Furthermore $\ol e\perp\ol E$, so $P_E=P_{\ol E}+P_{\ol e}$. Hence
$$
P_Ee
=
P_{\ol E}e+P_{\ol e}e
=
0+\frac{\langle e,\ol e\rangle}{|\ol e|^2} \ol e
=
\ol e.
$$
So Lemma~\ref{lem:fibration} applies, with $\ell_\xi$ and $L^{\pm}_\xi$ as defined  there. Set $\xi_0:=P_{\ol E}x_0$, which lies in $\inte(P_{\ol E}\Omega)$ since $x_0\in\Omega$. By Lemma~\ref{lem:affinity}, $\ell_{\xi_0}(t)$ is affine.

Next we compute $\ell_{\xi_0}(t)$. For $x\in e^\perp$, $P_{\ol E}x=\xi_0$ if and only if $x=x_0+v$ for some $v\in V$. Furthermore, $P_{\ol e}x=-\langle a,x\rangle\ol e$, since $\langle x,\ol e\rangle/|\ol e|^2=\langle x,e-a\rangle=-\langle a,x\rangle$. Hence, by Lemma~\ref{lem:fibration},
$$
L^+_{\xi_0}(t)=
-\langle a,x_0\rangle
+
\sup_{v\in V,\;x_0+v\in\cl(\Omega)}
\bigl\{r(x_0+v)-\langle a,v\rangle+t\,c(x_0+v)\bigr\}.
$$
Since $\ell_{\xi_0}=\bigl(L^+_{\xi_0}-L^-_{\xi_0}\bigr)|\ol e|$ is affine, while $L^+_{\xi_0}$ is convex and $L^-_{\xi_0}$ is concave by Lemma~\ref{lem:fibration}, both $L^{\pm}_{\xi_0}$ are affine. 
By the first inequality in \eqref{eq:av}, $r(x_0+v)-\langle a,v\rangle\leq r(x_0)$;
 so 
 $
 L^+_{\xi_0}(0)=r(x_0)-\langle a,x_0\rangle.
 $
 On the other hand, taking $v=0$ in the supremum gives 
 $
 L^+_{\xi_0}(t)\geq r(x_0)+tc(x_0)-\langle a,x_0\rangle.
 $ 
Since $L^+_{\xi_0}$ is affine and equality holds at the interior point $t=0$, it follows that 
 $$
 L^+_{\xi_0}(t)=r(x_0)+tc(x_0)-\langle a,x_0\rangle.
 $$

Comparing the last two displayed expressions for $L^+_{\xi_0}$ yields, for every $v\in V$ with $x_0+v\in\cl(\Omega)$, and every $-1\leq t\leq1$, 
$$
r(x_0+v)-r(x_0)-\langle a,v\rangle
+t\bigl(c(x_0+v)-c(x_0)\bigr)
\leq0.
$$
Set $v=\lambda v_0$, where $\lambda>0$ is small, divide by $\lambda$, and let $\lambda\to0^+$. Then the terms involving $r$ converge to $r'(x_0;v_0)-\langle a,v_0\rangle$, which vanishes by \eqref{eq:av}, and we obtain
$
t\,c'(x_0;v_0) \leq0.
$
Setting $t=\pm1$ yields $c'(x_0;v_0)=0$, as desired.
\end{proof}

\section{First Variation of the Quermassintegrals}\label{appendix:formula}
Here we show that \eqref{eq:key} holds without the strict
convexity assumption, which yields a more direct proof of Theorem \ref{thm:main}. The idea is to differentiate the volumes of the
parallel bodies $K_t+\rho B$, and then recover $W_m'(1)$ from
Steiner's formula by polynomial interpolation.
Put $s:=t-1$, so
$$
K_{1+s}=\{p+s\,c(\pi p)e\mid p\in K\}.
$$
Let 
$d_K\colon\R^{n+1}\to\R$ denote the distance function of $K$, and $P_K\colon\R^{n+1}\to K$ be the nearest point
projection.
 Recall that
$f_K(p):=c(\pi p)\langle e,N(p)\rangle$.  The key observation is that, even though the deformation $t\mapsto K_t$ 
is not smooth, its effect on the distance function is differentiable at 
$t=1$, with derivative $-f_K\circ P_K$:

\begin{lemma}\label{lem:distancevariation}
As $s\to0^-$,
$$
d_{K_{1+s}}
=
d_K-s\,(f_K\circ P_K)+o(|s|),
$$
uniformly on compact subsets of $\R^{n+1}\setminus K$.
\end{lemma}
\begin{proof}
It suffices to establish the expansion uniformly for $d_K(q)\in I$, 
where $I\subset(0,\infty)$ is an arbitrary compact interval.
Let $C:=\sup_{\pi K}|c|$. Since each
chord of $K$ is translated by at most $C|s|$, the Hausdorff distance
$\dist_H(K_{1+s},K)\leq C|s|$.
Fix $q$ with $d_K(q)\in I$, put $p:=P_Kq$, and choose $p_s\in K$ so
that $d_{K_{1+s}}(q)=|q-p_s-s\,c(\pi p_s)e|$. Then
$$
d_K(q)
\leq
|q-p_s|
\leq
d_{K_{1+s}}(q)+C|s|
\leq
d_K(q)+2C|s|.
$$
Since $K$ is convex and $p$ is its nearest point to $q$,
$\langle q-p,p_s-p\rangle\leq0$. Hence
$|q-p_s|^2\geq d_K(q)^2+|p_s-p|^2$, and consequently
$|p_s-p|^2\leq4C|s|(d_K(q)+C|s|)$. Thus $p_s\to p$ uniformly for
$d_K(q)\in I$.
Set 
$$
\nu_s:=(q-p_s)/|q-p_s|
\qquad \text{and} \qquad
\nu:=(q-p)/d_K(q)=N(p).
$$
 Then $\nu_s\to\nu$ uniformly. Moreover,
$c(\pi p_s)\langle e,\nu_s\rangle\to f_K(p)$ uniformly. Otherwise
there would be $\epsilon>0$, $s_i\to0^-$, and points $q_i$ with
$d_K(q_i)\in I$ for which the convergence fails. After passing to a
subsequence, $q_i\to q$, and hence $p_i:=P_Kq_i\to p:=P_Kq$,
$p_{s_i}\to p$, and $\nu_{s_i}\to N(p)$. If
$\langle e,N(p)\rangle=0$, boundedness of $c$ and continuity of $N$
show that both
$c(\pi p_{s_i})\langle e,\nu_{s_i}\rangle$ and $f_K(p_i)$ tend to
zero. Otherwise $\pi p\in\Omega$, where $c$ is continuous, and both
quantities tend to $f_K(p)$, a contradiction.
For vectors $a$ bounded away from the origin and bounded vectors $b$,
$|a-sb|=|a|-s\langle b,a/|a|\rangle+O(s^2)$ uniformly. Consequently,
\begin{gather*}
d_{K_{1+s}}(q)
\leq
|q-p-s\,c(\pi p)e|
=
d_K(q)-s f_K(p)+O(s^2),\\
d_{K_{1+s}}(q)
=
|q-p_s|-s\,c(\pi p_s)\langle e,\nu_s\rangle+O(s^2)
\geq
d_K(q)-s f_K(p)+o(|s|),
\end{gather*}
which completes the proof, since $O(s^2)=o(|s|)$.
\end{proof}

Now fix $\rho>0$ small enough that 
$\Phi(p,\tau):=p+\tau N(p)$
is a $\C^1$ diffeomorphism from $M\times(0,2\rho)$ to the tubular neighborhood
$T:=\{q\mid 0<d_K(q)<2\rho\}$ of $M_\rho:=\Phi(M,\rho)=\partial (K+\rho B)$. The Jacobian of $\Phi$ is
$
J(p,\tau)=\prod_{i=1}^n\bigl(1+\tau\kappa_i(p)\bigr),
$
where $\kappa_i$ are the principal curvatures of $M$. For every $q$, 
$$
|d_{K_{1+s}}(q)-d_K(q)|\leq\dist_H(K_{1+s},K)\leq C|s|.
$$
 Hence the 
symmetric difference of $K_{1+s}+\rho B$ and $K+\rho B$ lies in 
$\{|d_K-\rho|\leq C|s|\}$, which is contained in $T$ for small $s$. So 
we may compute the volume of this region in the coordinates $\Phi$, 
where $K+\rho B$ is given by $\tau\leq\rho$. Furthermore, by 
Lemma~\ref{lem:distancevariation}, there are $\epsilon_s\to0$ such that
$$
\bigl|d_{K_{1+s}}(\Phi(p,\tau))-\tau+s f_K(p)\bigr|
\leq
\epsilon_s|s|
$$
for $p\in M$, $\rho/2\leq\tau\leq3\rho/2$, and all sufficiently small 
$s<0$. Hence, if we set $\rho_s(p):=\rho+s f_K(p)$, then we have
$$
\bigl\{\tau\leq\rho_s(p)-\epsilon_s|s|\bigr\}
\subset
K_{1+s}+\rho B
\subset
\bigl\{\tau\leq\rho_s(p)+\epsilon_s|s|\bigr\}
$$
within $T$.  Therefore
\begin{multline*}
\int_M\int_\rho^{\rho_s(p)-\epsilon_s|s|}J(p,\tau)\,d\tau\,dA(p)
\leq
|K_{1+s}+\rho B|-|K+\rho B|\\
\leq
\int_M\int_\rho^{\rho_s(p)+\epsilon_s|s|}J(p,\tau)\,d\tau\,dA(p).
\end{multline*}
Dividing by $s<0$ reverses the inequalities, and both bounds converge
to the same limit. Hence
\begin{equation}\label{eq:parallelvariation}
\lim_{s\to0^-}
\frac{|K_{1+s}+\rho B|-|K+\rho B|}{s}
=
\int_M f_K(p)
\prod_{i=1}^n\bigl(1+\rho\kappa_i(p)\bigr)\,dA(p).
\end{equation}
Since
$\prod_{i=1}^n(1+\rho\kappa_i)=\sum_{k=0}^n\sigma_k\rho^k$, where
$\sigma_0:=1$, the right hand side of \eqref{eq:parallelvariation} is
a polynomial in $\rho$ whose coefficient of $\rho^k$ is
$\int_M f_K\sigma_k\,dA$. On the other hand, Steiner's formula \eqref{eq:steiner} gives
$$
\frac{|K_{1+s}+\rho B|-|K+\rho B|}{s}
=
\sum_{k=0}^n
\binom{n+1}{k}
\frac{W_k(K_{1+s})-W_k(K)}{s}\rho^k.
$$
Evaluating at $n+1$ distinct sufficiently small values of $\rho$ and
letting $s\to0^-$ shows that the coefficients converge. Comparing the
coefficient of $\rho^m$ yields
$$
W_m'(1)
=
\frac{1}{\binom{n+1}{m}}\int_M f_K\,\sigma_m\,dA.
$$
Finally,
$1/\binom{n+1}{m}=\ol{m}/((n+1)\binom{n}{m})$, while the change of
variables over the graphs of $u^\pm$, as in the proof of
Proposition~\ref{prop:strict}, gives
$\int_M f_K\sigma_m\,dA
=
\int_\Omega c(x)(\sigma_m^+(x)-\sigma_m^-(x))\,dx$.
Thus \eqref{eq:key} holds for every $\C^2$ convex body $K$.

\section*{Acknowledgement}
The author thanks Renato Bettiol for bringing Li's conjecture to his attention. Thanks also to Yanyan Li and Emanuel Milman for useful comments.
The AI tools Claude (Anthropic) and ChatGPT (OpenAI) were used during the preparation of this manuscript. The
author has reviewed all AI-assisted content and takes full responsibility for the final manuscript.

\bibliography{references}

\end{document}